\documentclass[a4paper,9pt]{article}
\usepackage{amsmath, amsthm, amsfonts}
\usepackage{amssymb}
\usepackage{latexsym}
\usepackage{verbatim}

\usepackage[all,cmtip]{xy}
\usepackage{hyperref}

\input{xypic}

\usepackage{enumerate}
\usepackage{amssymb}
\usepackage{array}

\DeclareMathAlphabet{\mathfr}{U}{euf}{m}{n}

\newtheorem{teo}{Theorem}[section]
\newtheorem{defi}{Definition}[section]
\newtheorem{rem}{Remark}[section]

\newtheorem{lema}{Lemma}[section]
\newtheorem{coro}{Corollary}[section]
\newtheorem{prop}{Proposition}[section]

\newcommand{\Q}{\mathbb Q}
\newcommand{\Qbar}{\overline{\mathbb Q}}
\newcommand{\Z}{\mathbb Z}
\newcommand{\M}{\mathbb{M}}

\DeclareMathOperator{\Gal}{Gal}
\DeclareMathOperator{\GL}{GL}
\DeclareMathOperator{\End}{End}
\DeclareMathOperator{\Hom}{Hom}
\DeclareMathOperator{\Frob}{Frob}
\DeclareMathOperator{\Aut}{Aut}
\DeclareMathOperator{\id}{id}
\DeclareMathOperator{\GAP}{GAP}
\DeclareMathOperator{\Ind}{Ind}

\DeclareMathOperator{\Isom}{Isom}
\DeclareMathOperator{\I}{I}
\DeclareMathOperator{\II}{II}
\DeclareMathOperator{\Res}{Res}
\DeclareMathOperator{\Twist}{Twist}
\DeclareMathOperator{\Tr}{Tr}

\DeclareMathOperator{\sgn}{sgn}
\numberwithin{equation}{section}

\title{The twisting representation of the\\ $L$-function of a curve}
\author{Francesc Fit\'e and Joan-C. Lario}
\date{\today}

\begin{document}
\maketitle

\begin{abstract}
Let $C$ be a smooth projective curve defined over a number field and let $C'$ be a twist of $C$. In this article we relate the $\ell$-adic representations attached to the $\ell$-adic Tate modules of the Jacobians of~$C$ and~$C'$ through an Artin representation. This representation induces \emph{global} relations between the local factors of the respective Hasse-Weil $L$-functions. We make these relations explicit in a particularly illustrating situation. For every $\Qbar$-isomorphism class of genus~2 curves defined over~$\Q$ with $\Aut(C)\simeq D_8$ or $D_{12}$, except for a finite number, we choose a representative curve $C/\Q$ such that, for every isomorphism $\phi\colon C'\rightarrow C$ satisfying some mild condition, we are able to determine either the local factor $L_{ p}(C'/\Q,T)$ or the product $L_{ p}(C'/\Q,T)\cdot L_{ p}(C'/\Q,-T)$ from the local factor $L_{p}(C/\Q,T)$.
\end{abstract}

\section{Introduction}

Let $C$ and $C'$ be smooth projective curves of genus $g\geq 1$ defined over a number field~$k$ that become isomorphic over an algebraic closure of $k$ (that is, they are \emph{twisted} of each other). The aim of this article is to relate the $\ell$-adic representations attached to the $\Q_\ell$-vector spaces $V_\ell(C)$  and $V_\ell(C')$. Here, for a prime~$\ell$, $V_\ell(C)$ stands for $\Q_\ell\otimes T_\ell(C)$, where $T_\ell(C)$ denotes the $\ell$-adic Tate module of the Jacobian variety $J(C)$ attached to $C$ (and analogously for $C'$).

The case of quadratic twists of elliptic curves is well known. If $E$ and $E'$ are elliptic curves defined over $k$ that become isomorphic over a quadratic extension $L/k$, then there exists a character $\chi$ of $\Gal(L/k)$ such that
\begin{equation}\label{tramp}
V_\ell(E')\simeq \chi\otimes V_\ell(E)\,.
\end{equation}
This translates into a relation of local factors of the corresponding Hasse-Weil $L$-functions. Indeed, one has that for every prime $\mathfrak p$ of $k$ unramified in $L$
\begin{equation}\label{equation: local factors}
L_\mathfrak p(E'/k,T)=L_\mathfrak p(E/k,\chi(\Frob_{\mathfrak p})T)\,.
\end{equation}

So from now on, we will assume that the genus of $C$ (and $C'$) is $g\geq 2$, and we will focus on obtaining a generalization of relation (\ref{tramp}).

Let us fix some notation. Hereafter, $\overline \Q$ denotes a fixed algebraic closure of $\Q$ that is assumed to contain $k$ and all of its algebraic extensions. For any algebraic extension $F/k$, we will write $G_F:=\Gal(\overline \Q/F)$. For abelian varieties $A$ and $B$ defined over $k$, denote by $\Hom_F(A,B)$ the $\Z$-module of homomorphisms from $A$ to $B$ defined over $F$, and by $\End_F(A)$ the ring of endomorphisms of $A$ defined over~$F$. Write $\Hom_F^0(A,B)$ for the $\Q$-vector space $\Q\otimes\Hom_F(A,B)$, and $\End_F^0(A)$ for the algebra $\Q\otimes\End_F(A)$. We write $A\sim_F B$ to denote that $A$ and $B$ are isogenous over $F$.

\subsection{Relating $\ell$-adic representations of twisted curves}

Let $\Aut(C)$ be the group of automorphisms defined over $\Qbar$ of $C$, and let $\Isom(C',C)$ be the set of all isomorphisms from $C'$ to $C$. Throughout the paper, $L/k$ (resp., $K/k$) will denote the minimal extension of $k$ where all the elements in $\Isom(C',C)$ (resp., in $\Aut(C)$) are defined. A theorem of Hurwitz asserts that $\Aut(C)$ has order less or equal than $84(g-1)$. Since the isomorphism $\phi$ induces a bijection between $\Aut(C)$ and $\Isom(C',C)$, in particular, we have that these two sets are finite. Thus, the extensions $K/k$ and $L/k$ are finite. Since the curves $C$ and $C'$ are defined over $k$, the extensions $K/k$ and $L/k$ are Galois extensions. Clearly, $K/k$ is a subextension of $L/k$. We can now state the principal result of Section~2.

\begin{teo}
The representation
$$
\theta_C\colon G_C:=\Aut(C)\rtimes_{\lambda_C}\Gal(K/k)\rightarrow \Aut_\Q(\End_K^0(J(C)))\,,
$$
defined by equation (\ref{equation: twist rep}) and
called the \emph{twisting representation of $C$}, satisfies that, for every $\theta_C$-twist $\phi\colon C'\rightarrow C$, there is an inclusion of $\Q_\ell[G_k]$-modules
\begin{equation}\label{equation: generalization}
V_\ell(C')\subseteq (\theta_C\circ\lambda_\phi)\otimes V_\ell(C)\,.
\end{equation}
Here $\lambda_\phi\colon\Gal(L/k)\rightarrow G_C$ stands for the monomorphism defined by equation (\ref{equation: emb}).
\end{teo}

This result encompasses Remark~\ref{remark: Fit10}, Proposition~\ref{tetator} and Theorem~\ref{factorteta}, and we refer to the remaining results of Section~$2$ for proofs of the well-definition of the objects involved in the statement. Requiring a twist~$C'$ of~$C$ to be a~$\theta_C$-twist is a mild condition that we precise in Definition~\ref{defi: theta_C twist}. In Proposition~\ref{gengen}, we show that~(\ref{equation: generalization}) indeed generalizes~(\ref{tramp}).

\subsection{Applications}

In the particular cases that we will look at, one can in fact compute the whole decomposition of $(\theta_C\circ\lambda_\phi)\otimes V_\ell(C)$. This leads to a relation between local factors of $C$ and local factors of $C'$ of the style of (\ref{equation: local factors}), that is, a relation written in terms of an Artin representation. Such kind of \emph{global} relations have been proved to be most useful when one is interested in the study of the behaviour of the local factor at a varying prime (e.g. generalized Sato-Tate distributions; see Section 4 of \cite{Fit10} and especially \cite{FS12}).

The essential feature of the cases considered in which one can perform the computation of the decomposition of $(\theta_C\circ\lambda_\phi)\otimes V_\ell(C)$ is the splitting of the Jacobian $J(C)$ over $K$ as the power of an elliptic curve $E/K$ (what we call the completely splitted Jacobian case). In this article we restrict to the case in which~$E$ does not have complex multiplication (CM), and we refer to~\cite{FS12} for a treatment of the case in which~$E$ has CM.

After the considerations of general type for the completely splitted Jacobian case of Section~3, we restrict our attention in Section~4 to the situation in which~$C$ is a genus~2 curve defined over $\Q$ with $\Aut(C)\simeq D_8$ (resp. $D_{12}$). Recall that every such a curve is $\overline \Q$-isomorphic to a curve $C_u$ in the family of (\ref{eqd8}) (resp. in the family of (\ref{eqd12})) for some $u$ in $\Q^* \smallsetminus\{1/4,9/100\}$ (resp. in $\Q^*\smallsetminus\{1/4,-1/50\}$). We then prove the following result.

\begin{teo}\label{theorem: application} Let $\phi:C'\rightarrow C$ be a twist of $C=C_u$  with $\Aut(C)\simeq D_8$ (resp. $\Aut(C)\simeq D_{12}$). Assume that $u$ does not belong to the finite list (\ref{uCMD8}) (resp. (\ref{uCMD12})). If $V_\ell(C')$ is a simple $\Q_\ell[G_K]$-module, then for every prime $p$ unramified in $L/\Q$, we have
$$
L_{ p}(C/\Q,\theta_C\circ \lambda_\phi,T)=
\begin{cases}
L_{ p}(C'/\Q,T)^4 & \text{if $f=1$}\\
L_{ p}(C'/\Q,T)^2L_{ p}(C'/\Q,-T)^2 & \text{if $f=2$,}
\end{cases}
$$
where $f$ denotes the residue class degree of~$p$ in~$K$.
\end{teo}
In the statement of the theorem, $L_{ p}(C/\Q,\theta_C\circ \lambda_\phi,T)$ stands for the Rankin-Selberg polynomial whose roots are all the products of roots of $L_{ p}(C/\Q,T)$ and roots of $\det(1-\theta_C\circ\lambda_\phi(\Frob_{ p})T)$.

\section{The twisting representation $\theta_C$}

For any twist $ C'$ of a smooth projective curve $C$ defined over $k$ of genus~$g\geq 2$, let $K/k$ and $L/k$ be as in the Introduction. We will write the natural action of the group $\Gal(L/k)$ on $\Aut(C)$, $\Isom(C',C)$, $\End_L^0(J(C))$, and $\Hom_L^0(J(C),J(C'))$ using left exponentiation and we will often avoid writing $\circ$ for the composition of maps. Then, we have the following monomorphism of groups
$$\lambda_C\colon\Gal(K/k)\rightarrow\Aut(\Aut(C)),\qquad \lambda_C(\sigma)(\alpha)={}^{\sigma}\alpha\,.$$
Indeed, the minimality of $K$ guarantees that if $\sigma\in \Gal(K/k)$ is such that $\alpha={}^{\sigma}\alpha$ for every
$\alpha\in\Aut(C)$, then $\sigma$ is trivial. We define the twisting group of~$C$ as
$$G_C:=\Aut(C)\rtimes_{\lambda_C}\Gal(K/k)\,,$$
where $\rtimes_{\lambda_C}$ denotes the semidirect product through the morphism~$\lambda_C$. We now proceed to somehow justify the name of $G_C$. First, we fix some notation. Suppose that $F'/k$ is a Galois extension and that $F/k$ is a Galois subextension of $F'/k$. Then, let $\pi_{F'/F}\colon\Gal(F'/k)\rightarrow\Gal(F/k)$ stand for the canonical projection. For every isomorphism $\phi\colon C'\rightarrow C$, define the map
\begin{equation}\label{equation: emb}
\lambda_\phi\colon\Gal(L/k)\rightarrow G_C,\qquad \lambda_\phi(\sigma)=(\phi(^{\sigma}\phi)^{-1},\pi_{L/K}(\sigma))\,.
\end{equation}
\begin{lema}\label{fields} The map $\lambda_\phi$ is a monomorphism of groups.
\end{lema}

\begin{proof} Let $\sigma$ and $\tau$ belong to $\Gal(L/k)$. Then, we have
$$
\begin{array}{l@{\,=\,}l}
\lambda_\phi(\sigma\tau) & \displaystyle{(\phi(^{\sigma\tau}\phi)^{-1},\pi_{L/K}(\sigma\tau))}\\[6pt]
 & \displaystyle{(\phi(^{\sigma}\phi)^{-1}\circ{}^{\sigma}(\phi(^{\tau}\phi)^{-1}), \pi_{L/K}(\sigma\tau))}\\[6pt]
 & \displaystyle{(\phi(^{\sigma}\phi)^{-1}\lambda_C(\pi_{L/K}(\sigma))(\phi(^{\tau}\phi)^{-1}), \pi_{L/K}(\sigma)\circ\pi_{L/K}(\tau))}\\[6pt]
 & \displaystyle{(\phi(^{\sigma}\phi)^{-1}, \pi_{L/K}(\sigma))(\phi(^{\tau}\phi)^{-1}, \pi_{L/K}(\tau))=\lambda_\phi(\sigma)\circ\lambda_\phi(\tau)}\,.\\[6pt]
\end{array}
$$
Let $\sigma\in\Gal(L/k)$ be such that $\phi(^{\sigma}\phi)^{-1}=\id$ and $\pi_{L/K}({\sigma})$ is trivial, i.e., $\phi={}^{\sigma}\phi$ and $\sigma\in\Gal(L/K)$. Let $\psi$ be any element of $\Isom(C',C)$. Since $\psi\phi^{-1}$ is an element of $\Aut(C)$, it is fixed by $\sigma$. Then, one has
$${}^{\sigma}\psi={}^{\sigma}(\psi\phi^{-1}\phi)={}^{\sigma}(\psi\phi^{-1}){}^{\sigma}\phi=\psi\phi^{-1}\phi=\psi\,.$$
The minimality of $L$ guarantees now that $\sigma$ is trivial.
\end{proof}

\begin{prop} A one-to-one correspondence between the elements of the following sets:
\begin{enumerate}[i)]
\item The set $\Twist(C/k)$ of twists of $C$ up to $k$-isomorphism;
\item The set of monomorphisms $\lambda\colon \Gal(F/k)\rightarrow G_C$ of the form $\lambda=\xi\rtimes_{\lambda_C}\pi_{F/K}$, with $\xi$ a map from $\Gal(F/k)$ to $\Aut(C)$, where we identify
$$\lambda_1\colon\Gal(F_1/k)\rightarrow G_C\qquad\text{and}\qquad\lambda_2\colon\Gal(F_2/k)\rightarrow G_C$$
if there exists $\alpha\in\Aut(C)$ such that, for every $\sigma\in \Gal(F_1F_2/k)$, one has $$\lambda_1\circ\pi_{F_1F_2/F_1}(\sigma)(\alpha,1)=(\alpha,1)\lambda_2\circ\pi_{F_1F_2/F_2}(\sigma)\,;$$
\end{enumerate}
is given by associating to a twist $C'$ of $ C$ the class of the monomorphism~$\lambda_\phi$, where~$\phi$ is any isomorphism from $C$ to $C'$.
\end{prop}
\begin{proof} There is a well-known bijection between the elements of $\Twist(C/k)$ and the elements of the cohomology set $H^1(G_k,\Aut(C))$, given by associating to a twist $ C'$ of $C$ the class of the cocycle $\xi(\sigma)=\phi(^{\sigma}\phi)^{-1}$ (see \cite{Sil86}, chapter~X). Now, associate to the cocycle $\xi$, the morphism $\tilde \lambda\colon G_k\rightarrow G_C$, defined by $\tilde\lambda=\xi\rtimes_{\lambda_C}\pi_{\overline k/K}$. Observe that, for $\sigma,\,\tau$ in $G_k$, one has that $\tilde\lambda(\sigma\tau)=\tilde\lambda(\sigma)\tilde\lambda(\tau)$ if and only if
$\xi(\sigma\tau)=\xi(\sigma)\circ{}^{\sigma}\xi(\tau)$. Let $G_F$ denote the kernel of $\tilde\lambda$ and let $\lambda\colon \Gal(F/k)\rightarrow G_C$ satisfy $\tilde\lambda=\lambda\circ\pi_{\overline k/F}$. Then $\lambda$ is injective. Moreover, the cocycles $\xi_1$ and $\xi_2$ are cohomologous if and only if there exists $\alpha$ in $\Aut(C)$ such that for all $\sigma$ in $G_k$ it holds $\xi_1(\sigma)\circ{}^{\sigma}\alpha=\alpha\circ\xi_2(\sigma)$, which is equivalent to  $\tilde\lambda_1(\sigma)(\alpha,1)=(\alpha,1)\tilde\lambda_2(\sigma)$. Finally, this amounts to ask that $\lambda_1\circ\pi_{F_1F_2/F_1}(\sigma)(\alpha,1)=(\alpha,1)\lambda_2\circ\pi_{F_1F_2/F_2}(\sigma)$ for every $\sigma\in \Gal(F_1F_2/k)$.
\end{proof}

\begin{prop} The monomorphism $\lambda_\phi$ is an isomorphism if and only if the action of $\Gal(L/K)$ on $\Isom(C',C)$ has a single orbit.
\end{prop}
\begin{proof} One has that $\lambda_\phi$ is exhaustive if and only if $|\!\Aut(C)|=|\Gal(L/K)|$. This is equivalent to the fact that the injective morphism
$$\lambda\colon\Gal(L/K)\rightarrow \Aut(C),\,\qquad \lambda(\sigma)=\phi(^\sigma\phi)^{-1}$$
is an isomorphism. This happens if and only if for every $\alpha \in\Aut(C)$ there exists $\sigma \in\Gal(L/K)$ such that $\alpha\phi={}^\sigma\phi$, that is, if and only if for every $\psi\in\Isom(C',C)$, there exists $\sigma \in\Gal(L/K)$ such that $\psi={}^\sigma\phi$.
\end{proof}

\begin{rem}\label{remark: Fit10} For any twist $C'$ of $C$, the abelian varieties $J(C)$ and $J(C')$ are defined over~$k$ and are isogenous over $L$. Let $F/k$ be a subextension of $L/k$. Denote by $\theta(C,C';L/F)$ the representation afforded by the $\Q[\Gal(L/F)]$-module $\Hom_L^0(J(C),J(C'))$. We will write $\theta(C,C'):=\theta(C,C';L/k)$. We recall that Theorem 3.1 of \cite{Fit10} asserts that
$$
V_\ell(C')\subseteq \theta(C,C')\otimes V_\ell(C)
$$
as $\Q_\ell[G_k]$-modules.
\end{rem}

Every isomorphism $\phi$ from $C'$ to $C$ induces an isomorphism from $J(C')$ to $J(C)$, that we will also call $\phi$. Consider the map
$$\theta_\phi\colon \Gal(L/k)\rightarrow \Aut_\Q(\End_L^0(J(C)))\,,\qquad \theta_\phi(\sigma)(\psi)=\phi(^\sigma\phi)^{-1}\circ{}^\sigma\psi\,,$$
where $\sigma$ is in $\Gal(L/k)$ and $\psi$ in $\End_L^0(J(C))$.

\begin{prop}\label{tetator} For every isomorphism $\phi\colon C'\rightarrow C$, the map $\theta_\phi$ is a rational representation of $\Gal(L/k)$ isomorphic to $\theta(C,C')$.
\end{prop}

\begin{proof}
It is indeed a representation. For $\sigma$ and $\tau$ in $\Gal(L/k)$, one has
$$
\begin{array}{l@{\,=\,}l}
\theta_\phi(\sigma\tau)(\psi) & \displaystyle{\phi(^{\sigma\tau}\phi)^{-1}\circ{}^{\sigma\tau}\psi}\\[6pt]
 & \displaystyle{\phi(\phi^{\sigma})^{-1}\circ{}^{\sigma}(\phi({}^{\tau}\phi)^{-1}\circ{}^{\tau}\psi)}\\[6pt]
 & \displaystyle{(\theta_\phi(\sigma)\circ\theta_\phi(\tau))(\psi)\,.}
\end{array}
$$
The map $\tilde\phi\colon \Hom_L^0(J(C),J(C'))\rightarrow \End_L^0(J(C))$, defined by $\tilde\phi(\varphi)=\phi\circ\varphi$ for $\varphi\in\Hom_L^0(J(C),J(C'))$ is an isomorphism of $\Q$-vector spaces. Now, one deduces that $\theta(C,C')$ and $\theta_\phi$ are isomorphic from the fact that, for every $\sigma$ in $\Gal(L/k)$, the following diagram is commutative
$$
\xymatrix{
 \Hom_L^0(J(C),J(C'))\ar[d]_{\tilde\phi} \ar[rr]^{\theta(C,C')(\sigma)} && \Hom_L^0(J(C),J(C'))\ar[d]^{\tilde\phi}\\
\End_L^0(J(C)) \ar[rr]^{\theta_\phi(\sigma)} &&\End_L^0(J(C))\,.}
$$
\end{proof}

Denote also by $\alpha$ the induced endomorphism in $J(C)$ by an automorphism $\alpha$ in $\Aut(C)$. We define the twisting representation of the $L$-function of $C$ as the map
\begin{equation}\label{equation: twist rep}
\theta_C\colon G_C\rightarrow\Aut_\Q(\End_K^0(J(C)))\,,\qquad\theta_C((\alpha,\sigma))(\psi)=\alpha\circ{}^{\sigma}\psi\,,
\end{equation}
where $\sigma$ in $\Gal(K/k)$ and $\psi$ in $\End_K^0(J(C))$.

\begin{defi}\label{defi: theta_C twist} We will say that a twist~$C'$ of~$C$ is a $\theta_C$-twist of $C$ if $L$ is such that $\End_K^0(J(C))=\End_L^0(J(C))$.
\end{defi}

\begin{teo}\label{factorteta} The map $\theta_C$ is a faithful representation of $G_C$. Moreover, for every $\theta_C$-twist $C'$ of $C$ and every isomorphism $\phi\colon C'\rightarrow C$, one has $\theta_C\circ\lambda_\phi=\theta_\phi$, that is, the following diagram is commutative
$$
\xymatrix{
 \Gal(L/k)\ar[rd]_{\theta_\phi} \ar@{^{(}->}[r]^{\lambda_\phi} & G_C\ar[d]^{\theta_C}\\
 &\Aut_\Q(\End_K^0(J(C)))\,.}
$$
\end{teo}

\begin{proof}
For $\psi_1,\,\psi_2$ in $\Aut(C)$ and $\sigma_1,\,\sigma_2$ in $\Gal(K/k)$, one has
$$
\begin{array}{l@{\,=\,}l}
\theta_C((\alpha_1,\sigma_1)(\alpha_2,\sigma_2))(\psi) & \displaystyle{\theta_C((\alpha_1\circ{}^{\sigma_1}\alpha_2,\sigma_1\sigma_2))(\psi) }\\[6pt]
 & \displaystyle{\alpha_1\circ{}^{\sigma_1}\alpha_2\circ^{\sigma_1\sigma_2}\psi}\\[6pt]
 & \displaystyle{\alpha_1\circ{}^{\sigma_1}(\alpha_2\circ^{\sigma_2}\psi)}\\[6pt]
 & \displaystyle{(\theta_C((\alpha_1,\sigma_1))\circ\theta_C((\alpha_2,\sigma_2)))(\psi)\,.}\\[6pt]
\end{array}
$$
Let $\alpha$ in $\Aut(C)$ and $\sigma$ in $\Gal(K/k)$ be such that $\theta_C(\alpha,\sigma)(\psi)=\psi$ for every $\psi$ in $\End_K^0(J(C))$. In particular, for $\psi=\alpha$, one obtains that ${}^\sigma\alpha=\id$, which implies $\alpha=\id$. Then $\psi={}^\sigma\psi$ for all $\psi$ in $\End_K^0(J(C))$ and the minimality of $K$ implies that $\sigma$ is trivial.
Finally, it holds
$$(\theta_C\circ\lambda_\phi)(\sigma)(\psi)=\theta_C(\phi({}^\sigma \phi)^{-1},\pi_{L/K}(\sigma))(\psi)=\phi({}^{\sigma}\phi)^{-1}\circ{}^{\sigma}\psi=\theta_\phi(\sigma)(\psi)\,,$$
for $\sigma$ in $\Gal(L/k)$ and $\psi$ in $\End_L^0(J(C))$.
\end{proof}

As a corollary of the previous results one obtains the desired inclusion
\begin{equation}\label{equation: prin}
V_\ell(C')\subseteq (\theta_C\circ\lambda_\phi)\otimes V_\ell(C)
\end{equation}
for every $\theta_C$-twist $C'$ of $C$. This inclusion is a generalization of the identity (\ref{tramp}).

\begin{prop}\label{gengen} If $C'$ is a nontrivial twist of $C$ such that $\End_L^0(J(C))\simeq\Q$, then the extension $L/k$ is quadratic, the representation $\theta_C\circ \lambda_\phi$ is the quadratic character of $\Gal(L/k)$, and one has
$V_\ell(C')\simeq (\theta_C\circ\lambda_\phi)\otimes V_\ell(C)$.
\end{prop}

\begin{proof} By the inclusion (\ref{equation: prin}), it is enough to prove that $L/k$ is quadratic and that $\theta(C,C')$ is the quadratic character of $L/k$. Since $\Aut(C)$ injects in $\End_L^0(J(C))=\End_k^0(J(C))\simeq\Q$, we have that $\Aut(C)$ injects in $C_2$ and that $K=k$. Since $C'$ is nontrivial, then $\Aut(C)$ is nontrivial and, by Lemma \ref{fields}, we deduce that $L/k$ is a quadratic extension. Since the $1$-dimensional representation $\theta(C,C')$ is faithful, it corresponds to the quadratic character of $\Gal(L/k)$.
\end{proof}

\section{The completely splitted Jacobian case}

In this section we explore the twisting representation $\theta_C$ when the Jacobian $J(C)$ splits over $K$ as the power $E^g$ of an elliptic curve $E$ defined over~$K$ without complex multiplication (CM). Note that in this case $\dim\theta_C=g^2$. We will use the notation $H_C=\Aut(C)$ when we see $\Aut(C)$ as a subgroup of the twisting group~$G_C$. For future use, we will be interested in the following cases:
\begin{itemize}
\item[(I)]  $[K\colon k]=g^2$, the elliptic curve $E$ does not have CM, and $\theta_C$ is absolutely irreducible.

\item[(II)] $[K\colon k]= g^2/2$, the elliptic curve $E$ does not have CM, and $\theta_C\simeq_{\overline \Q} \theta_1\oplus \theta_2$ for $\theta_1$ and $\theta_2$ absolutely irreducible non-isomorphic representations such that $\Res^{G_C}_{H_C}\theta_1=\Res^{G_C}_{H_C}\theta_2$.
\end{itemize}

\begin{lema}\label{restheta} Suppose that $J(C)\sim_K E^g$, for $E$ an elliptic curve defined over $K$ without CM. One has:
$$\Res^{G_C}_{H_C}\theta_C\simeq g\cdot\varrho\,,$$
where $\varrho$ is a rational representation of $H_C$ of dimension $g$.
\end{lema}

\begin{proof} Consider the isomorphism
$$\Phi\colon\End_K^0(J(C))\simeq\End_K^0(E^g)\rightarrow\bigoplus_{i=1}^g\Hom_K^0(E,E^g)\,,$$
defined by $\Phi(\varphi)=(\varphi\circ\iota_1,\dots,\varphi\circ\iota_g)$, where $\iota_i\colon E\rightarrow E^g$ is the inclusion of $E$ to the $i$-th component of $E^g$. The action of $H_C=\Aut(C)$, which is by right composition, clearly restricts to each $\Hom_K^0(E,E^g)$. The rational representation $\varrho$ afforded by  $\Hom_K^0(E,E^g)$ satisfies $\Res^{G_C}_{H_C}\theta_C\simeq g\cdot\varrho$, and has dimension $g$ provided that $E$ has no CM.
\end{proof}

\begin{prop}\label{guap} Suppose that $J(C)\sim_K E^g$, for~$E$ an elliptic curve defined over~$K$. Suppose we are either in case $(\I)$ or $(\II)$. Let $\varrho$ be as in Lemma~ \ref{restheta}. Then, one has
$$
\Ind_{H_C}^{G_C}\varrho\simeq\frac{[K:k]}{g}\cdot \theta_C.
$$
\end{prop}

\begin{proof}
Let $(\cdot,\cdot)_{G_C}$ and $(\cdot,\cdot)_{H_C}$ denote the scalar products on complex-valued functions on $G_C$ and $H_C$, respectively. For the case (I), by Frobenius reciprocity, the multiplicity of $\theta_C$ in $\Ind^{G_C}_{H_C}\varrho$~is
$$(\Tr\Ind^{G_C}_{H_C}\varrho,\Tr\theta_C)_{G_C}=(\Tr\varrho,\Tr\Res^{G_C}_{H_C}\theta_C)_{H_C}=g\cdot(\Tr\varrho,\Tr\varrho)_{H_C}\geq g\,.$$
Since $[K\colon k]=g^2$, the dimensions of $\Ind^{G_C}_{H_C}\varrho$ and $g\cdot\theta_C$ equal $g^3$, and the result follows.

For the case (II), observe that $\Res^{G_C}_{H_C}\theta_1=\Res^{G_C}_{H_C}\theta_2$ implies that $\Res^{G_C}_{H_C}\theta_1=g/2\cdot\varrho$. Then, the multiplicity of $\theta_1$ in $\Ind^{G_C}_{H_C}\varrho$ is
$$(\Tr\Ind^{G_C}_{H_C}\varrho,\Tr\theta_1)_{G_C}=(\Tr\varrho,\Tr\Res^{G_C}_{H_C}\theta_1)_{H_C}=\frac{g}{2}\cdot(\Tr\varrho,\Tr\varrho)_{H_C}\geq \frac{g}{2}\,,$$
from which one sees that $g/2\cdot\theta_1$ is a subrepresentation of $\Ind^{G_C}_{H_C}\varrho$. Analogously, one proves that $g/2\cdot\theta_2$ is a subrepresentation of $\Ind^{G_C}_{H_C}\varrho$. Therefore, $g/2\cdot\theta_C$ is a subrepresentation of $\Ind^{G_C}_{H_C}\varrho$ and, since they both have dimension equal to $g^3/2$, they are isomorphic.
\end{proof}

\begin{coro}\label{coroguap} Suppose that $J(C)\sim_K E^g$, for $E$~an elliptic curve defined over~$K$. Suppose we are either in case $(\I)$ or $(\II)$. Then, one has
$$
\Ind_{H_C}^{G_C}\Res_{H_C}^{G_C}\theta_C \simeq [K:k]\cdot \theta_C
$$
\end{coro}

In what follows we will be particularly interested in the structure of $V_\ell(C)$ as a $\Q_{\ell}[G_K]$-module. First, we settle the following notation.  For an isomorphism $\phi\colon C'\rightarrow C$, denote by
$$\Res\lambda_\phi\colon\Gal(L/K)\rightarrow\Aut(C)$$
the restriction of the morphism $\lambda_\phi$ to the subgroup $\Gal(L/K)$. Observe that $$\Res^{G_C}_{H_C}\theta_C\circ \Res\lambda_\phi\simeq\theta(C,C';L/K)\,.$$

\begin{teo}\label{supguap} Suppose that $J(C)\sim_K E^g$, for $E$ an elliptic curve defined over $K$. Let $C'$ be a $\theta_C$-twist of $C$. Suppose that $V_\ell(C')$ is a simple $\Q_\ell[G_K]$-module. Then, one has:

$$
\theta(C,C')\otimes V_\ell(C)\simeq
\begin{cases}
\Q[\Gal(K/k)]\otimes V_\ell(C') & \text{if $(\I)$,} \\[3pt]
2\cdot \Q[\Gal(K/k)]\otimes V_\ell(C') & \text{if $(\II)$.} \\[3pt]
\end{cases}
$$
\end{teo}

\begin{proof}
For the case (I), recall that by Theorem $3.1$ in \cite{Fit10} there is an inclusion of $\Q_\ell[G_K]$-modules
$$
\begin{array}{ll}
V_\ell(C') & \displaystyle{\subseteq\, \theta(C,C';L/K)\otimes V_\ell(C)}\\[6pt]
 & \displaystyle{\simeq\,(\Res^{G_C}_{H_C}\theta_C\circ\Res\lambda_\phi)\otimes V_\ell(C) }\\[6pt]
 & \displaystyle{\simeq\, g^2\cdot (\varrho\circ\Res\lambda_\phi)\otimes V_\ell(E)\,.}\\[6pt]
\end{array}
$$

Since $V_\ell(C')$ is a simple $\Q_\ell[G_K]$-module, we obtain that
\begin{equation}\label{isosposdos}
V_\ell(C')\simeq(\varrho\circ\Res\lambda_\phi)\otimes V_\ell(E)\,.
\end{equation}
Now, tensoring by $g\cdot\Q[\Gal(K/k)]$ on both sides of the previous isomorphism we get
$$
\begin{array}{l@{\,\simeq\,}l}
g\cdot \Q[\Gal(K/k)]\otimes V_\ell(C') & \displaystyle{g\cdot \Ind_{K}^{k}(\varrho\circ\Res\lambda_\phi)\otimes V_\ell(E)}\\[6pt]
& \displaystyle{\Ind_{K}^{k}(\varrho\circ\Res\lambda_\phi)\otimes V_\ell(C)}\\[6pt]
 & \displaystyle{(\Ind^{G_C}_{H_C}\varrho\circ\lambda_\phi)\otimes V_\ell(C)}\\[6pt]
 & \displaystyle{ g\cdot(\theta_C\circ\lambda_\phi)\otimes V_\ell(C) }\\[6pt]
 & \displaystyle{ g\cdot\theta_\phi\otimes V_\ell(C) }\\[6pt]
 & \displaystyle{ g\cdot\theta(C,C')\otimes V_\ell(C) \,,}\\[6pt]
\end{array}
$$
where we have used that $\Ind^{G_C}_{H_C}\varrho=g\cdot\theta_C$, as seen in Proposition~\ref{guap}.
For the case~(II), everything is analogous to case (I) until equation~(\ref{isosposdos}). Then, tensoring by $2g\cdot\Q[\Gal(K/k)]$, we get
$$
\begin{array}{l@{\,\simeq\,}l}
2g\cdot \Q[\Gal(K/k)]\otimes V_\ell(C') & \displaystyle{2g\cdot \Ind_{K}^{k}(\varrho\circ\Res\lambda_\phi)\otimes V_\ell(E)}\\[6pt]
& \displaystyle{2\Ind_{K}^{k}(\varrho\circ\Res\lambda_\phi)\otimes V_\ell(C)}\\[6pt]
 & \displaystyle{2(\Ind^{G_C}_{H_C}\varrho\circ\lambda_\phi)\otimes V_\ell(C)}\\[6pt]
 & \displaystyle{ g\cdot(\theta_C\circ\lambda_\phi)\otimes V_\ell(C) }\\[6pt]
 & \displaystyle{ g\cdot\theta(C,C')\otimes V_\ell(C) \,.}\\[6pt]
\end{array}
$$
\end{proof}

\begin{coro}\label{signe} Assume the same hypothesis of Theorem~ \ref{supguap} and that one of the cases $(\I)$ or $(\II)$ holds. Let $\mathfrak p$ a prime of good reduction for both $C$ and $C'$ unramified in $L/k$. Write $a_{\mathfrak p} =\Tr\varrho_C(\Frob_{\mathfrak p})$ and $a_{\mathfrak p}' =\Tr\varrho_{C'}(\Frob_{\mathfrak p})$. Then:
\begin{enumerate}[i)]
\item If $\Frob_{\mathfrak p}\in G_K$, one has
    $$\sgn(a_{\mathfrak p}\cdot\Tr(\theta(C,C')(\Frob_{\mathfrak p})))=\sgn( a_{\mathfrak p}')\,.$$
\item If $\Frob_{\mathfrak p}\not\in G_K$, one has
$$\Tr\theta(C,C')(\Frob_{\mathfrak p})=0\,.$$
\end{enumerate}
\end{coro}

\begin{proof}
Theorem \ref{supguap} implies
$$\Tr(\theta(C,C')(\Frob_{\mathfrak p}))\cdot a_{\mathfrak p}= a_{\mathfrak p}'\cdot \Tr(\Q[\Gal(K/k)](\Frob_{\mathfrak p}))\,.$$
Part $i)$ follows from the fact that if $\Frob_{\mathfrak p}\in G_K$, then $$\Tr(\Q[\Gal(K/k)](\Frob_{\mathfrak p}))=|\Gal(K/k)|\,.$$

For part $ii)$, suppose that $\Frob_{\mathfrak p}\not\in G_K$. Corollary \ref{coroguap} implies that $\Tr\theta_C(\sigma)=0$ for any $\sigma\not\in H_C$.
Then, $\Tr\theta(C,C') (\Frob_{\mathfrak p})=\Tr\theta_C\circ\lambda_\phi(\Frob_{\mathfrak p})=0$.
\end{proof}

\section{The genus $2$ case}

Throughout this section,  $C$ denotes a genus $2$ curve defined over $\Q$. Let us recall some basic facts that may be found in  \cite{CGLR99}. It is well known that $C$ admits an affine model given by a hyperelliptic equation  $Y^2=f(X)$, where $f(X)\in \Q[X]$. Any element $\alpha\in\Aut(C)$ can then be written in the form
$$\alpha(X,Y)=\left(\frac{mX+n}{pX+q},\frac{mq-np}{(pX+q)^3}Y\right)\,,$$
for unique $m,\,n\,,p\,,q\in K$. Moreover, the map
$$\Aut(C)\rightarrow \GL_2(K)\,,\qquad \alpha\mapsto
\begin{pmatrix}
  m & n \\
  p & q
\end{pmatrix}\,
$$
defines a $2$-dimensional faithful representation of $\Aut(C)$. We will often identify an automorphism of $C$ with its corresponding matrix. Note that $w(X,Y)=(X,-Y)$ is always an automorphism of $C$, called the hyperelliptic involution of~$C$, which lies in the center $Z(\Aut(C))$ of $\Aut(C)$.

The group $\Aut(C)$ is isomorphic to one of the groups
$$C_2,\,\,C_2\times C_2,\,\,D_8,\,\,D_{12},\,\,2D_{12},\,\,\tilde { S}_4,\,\, C_2\times C_5\,,$$
where $2D_{12}$ and $\tilde {S}_4$ denote certain double covers of the dihedral group of $12$ elements $D_{12}$ and the symmetric group on $4$ letters $S_4$. Completing the study initiated by Clebsch and Bolza, Igusa \cite{Igu60} computed the $3$-dimensional moduli variety $\mathcal M_2$ of genus~$2$ curves defined over $\overline \Q$. Generically, the only non-trivial automorphism of a curve in $\mathcal M_2$ is the hyperelliptic involution and, thus, $\Aut(C)\simeq C_2$. The curves with $\Aut(C)$ containing $C_2\times C_2$ constitute a surface in $\mathcal M_2$. The moduli points corresponding to curves such that $\Aut(C)$ contains $D_8$ or $D_{12}$ describe two curves contained in this surface. The curves with $\Aut(C)\simeq 2D_{12}$, $\tilde S_4$, or $C_2\times C_5$ correspond to three isolated points of $\mathcal M_2$.

In this section, we will explicitly compute the twisting representation~$\theta_C$ of~$C$ and the decomposition of $\theta(C,C')\otimes V_\ell(C)$  when $\Aut(C)\simeq D_8$ or $D_{12}$. In both cases, the irreducible characters of $G_C$ will be denoted $\chi_i$, even though they refer to different groups (we will always refer the reader to the corresponding character table in Section~\ref{Appendix}). We will denote by $\varrho_i$ a representation of character~$\chi_i$.

\begin{lema}\label{trenc} If $\Aut(C)$ is non-abelian, then $J(C)\sim_KE^2$, where $E$ is an elliptic curve defined over $K$.
\end{lema}

\begin{proof} It is straightforward to check that $\Aut(C)$ contains a non-hyperelliptic involution $u$. Then the quotient $E=C/\langle u\rangle$ is an elliptic curve defined over $K$ (see Lemmas 2.1 and 2.2 in \cite{CGLR99}). The injection $E\hookrightarrow J(C)$ is also defined over $K$ and Poincar\'e Decomposition Theorem ensures the existence of an elliptic curve $E'$ defined over $K$ such that $J(C)\sim_K E\times E'$. Since $\End_K(J(C))$ contains $\Aut(C)$, it is non-abelian and so  $\End_K(J(C))\simeq \M_2(\End_K(E))$, from which $E\sim_K E'$.
\end{proof}

\begin{rem}\label{uambcm} From now on, for the cases  $\Aut(C)\simeq D_8$ or $ D_{12}$, we will make the assumption that the elliptic quotient~$E$ does not have complex multiplication, i.e., $\End_K^0(J(C))\simeq \M_2(\Q)$. This only excludes a finite number of $\overline \Q$-isomorphism classes. Indeed, curves with $\Aut(C)\simeq D_8$ or $D_{12}$ defined over $\Q$ are parameterized by rational values of their absolute invariant~$u$ (see subsections \ref{subd8} and \ref{subd12} for the details). According to Proposition $8.2.1$ of \cite{Car01} the $j$-invariant of the elliptic quotient~$E$ has two possibilities
$$
j(E)=
\begin{cases}\frac{2^6(3\mp10\sqrt u)^3}{(1\mp2\sqrt u)(1\pm2\sqrt u)^2}& \text{if $\Aut(C)\simeq D_8$,}\\[6pt]
\frac{2^83^3(2\mp5\sqrt u)^3(\pm\sqrt u)}{(1\mp2\sqrt u)(1\pm 2\sqrt u)^3}& \text{if $\Aut(C)\simeq D_{12}$.}
\end{cases}
$$
Since the degree of the extension $\Q(j(E))/\Q$ is $1$ or $2$ and the number of quadratic imaginary fields of class number $1$ or $2$ is finite, we deduce that there exists only a finite number of rational absolute invariants $u$ for which $E$ has CM. According to the table on page 112 of \cite{Car01}, for $\Aut(C)\simeq D_{8}$ these values of $u$ are:
\begin{equation}\label{uCMD8}
\frac{81}{196}\,,\frac{3969}{16900}\,,\frac{-81}{700}\,,\frac{1}{5}\,,\frac{9}{32}\,,\frac{12}{49}\,,\frac{81}{320}\,,\frac{81}{325}\,,
\end{equation}
\begin{equation*}
\frac{2401}{9600}\,,\frac{9801}{39200}\,,\frac{6480}{25920}\,,\frac{194481}{777925}\,,\frac{96059601}{384238400}\,. \end{equation*}
For $\Aut(C)\simeq D_{12}$ the values of $u$ for which $E$ has CM are:
\begin{equation}\label{uCMD12}
\frac{4}{25}\,,\frac{-4}{11}\,,\frac{1}{20}\,,\frac{1}{2}\,,\frac{27}{100}\,,\frac{4}{17}\,,\frac{125}{484}\,,\frac{20}{81}\,,\frac{256}{1025}\,,\frac{756}{3025}\,,\frac{62500}{250001}\,.
\end{equation}
\end{rem}

\begin{rem}
By Lemma \ref{trenc}, if $\Aut(C)\simeq D_8$ or $D_{12}$, then for every twist~$C'$ of~$C$, one has that
$$\End_L^0(J(C))= \End_K^0(J(C))\simeq \M_2(\End_K(E))\,.$$
In other words, every twist $C'$ of $C$ is a $\theta_C$-twist of $C$.
\end{rem}

\subsection{$\Aut(C)\simeq D_8$}\label{subd8}
\begin{prop}[Proposition 2.1 of \cite{CQ07}] There is a bijection between the $\overline \Q$-isomorphism classes of genus $2$ curves defined over $\Q$ with $\Aut(C)\simeq D_8$ and the open set of the affine line
$\Q^*\smallsetminus\{1/4,9/100\}$, given by associating to each $u\in \Q^*\smallsetminus\{1/4,9/100\}$ the projective curve of equation
$$Y^2Z^3=X^5+X^3Z^2+uXZ^4\,.$$
\end{prop}
As follows from Proposition 4.4 of \cite{CQ07}, the curve in the previous proposition is $\overline \Q$-isomorphic to
\begin{equation}\label{eqd8}
C=C_u\colon Y^2Z^4=X^6-8X^5Z+\frac{3}{u}X^4Z^2+\frac{3}{u^2}X^2Z^4+\frac{8}{u^2}XZ^5+\frac{1}{u^3}Z^6\,.
\end{equation}
where we have chosen parameters $z=0$, $s=1$ and $v=1/u$. Its group of automorphisms is computed loc. cit. in Proposition 3.3, and it is generated by
$$
U=\begin{pmatrix}
1/\sqrt 2 & 1/\sqrt{ 2u}\\
\sqrt{u/ 2} & -1/\sqrt 2
\end{pmatrix}\,,
\qquad
V=\begin{pmatrix}
0 & -1/\sqrt u\\
\sqrt u & 0
\end{pmatrix}\,,
$$
from which we see that $K=\Q(\sqrt u,\sqrt 2)$. Note that $U$ and $V$ satisfy the relations $U^2=1$, $V^4=1$ and $UV=V^3U$. For the character table of the group $G_C$, see in Section \ref{Appendix} Table \ref{taud8} if $u$ and $2u\not\in \Q^{*2}$; Table \ref{taud8p} if $u\in \Q^{*2}$; and Table \ref{taud8pp} if $2u\in \Q^{*2}$.

\begin{prop}\label{thetad8} One has
$$\Tr\theta_C=
\begin{cases}\chi_{11} & \text{ if $u$ and $2u\not\in \Q^{*2}$,}\\
\chi_{9}+\chi_{10} & \text{ if $u\in \Q^{*2}$,}\\
\chi_{6}+\chi_{7} & \text{ if $2u\in \Q^{*2}$.}
\end{cases}
$$
Moreover, $\Res^{G_C}_{H_C} \chi_9=\Res^{G_C}_{H_C}\chi_{10}$ in the second case, and $\Res^{G_C}_{H_C} \chi_6=\Res^{G_C}_{H_C}\chi_{7}$ in the third case.
\end{prop}
\begin{proof} The dimension of $\theta_C$ is $4$. Suppose that $u$ and $2u\not\in \Q^{*2}$. By looking at the column of the conjugacy class $2A$ in Table \ref{taud8}, one sees that $\varrho_{11}$ is the only faithful representation of dimension $4$ of~$G_C$.

One can also directly compute the representation $\theta_C$. Denote by $\alpha^*$ the image of $\alpha\in \Aut(C)$ by the inclusion $\Aut(C)\hookrightarrow \End_K^0(J(C))$. We will prove that $\End^0_K(J(C))=\langle 1^*,\,U^*,\,V^*,\,U^*V^*\rangle_\Q\,.$ Indeed, it is enough to see that $1^*,\,U^*,\,V^*$ and $U^*V^*$ are linearly independent. Suppose that for certain $\lambda_i$ in $\Q$, one has
$\lambda_1 1^*+\lambda_2 U^*+\lambda_3 V^*+\lambda_4 U^*V^*=0\,.$
Conjugating by $V^*$ one obtains
$\lambda_11^*-\lambda_2 U^*+\lambda_3 V^*-\lambda_4 U^*V^*=0\,,$
which implies $\lambda_11^*+\lambda_3 V^*=0$ and thus $\lambda_1=\lambda_3=0$. Analogously, one has $\lambda_2 U^*+\lambda_4 U^*V^*=0$, that is $\lambda_21^*+\lambda_4 V^*=0$, which implies $\lambda_2=\lambda_4=0$.
Let $\sigma,\,\tau\in\Gal(K/\Q)$ be such that $\sigma(\sqrt u)=-\sqrt u$ and $\tau(\sqrt 2)=-\sqrt 2$. Now, $\theta_C$ can be computed by observing that
${}^\sigma U=UV$, ${}^{\sigma}V=V^3$, ${}^{\tau}U=UV$, ${}^{\tau}V=V$.

Suppose that $u\in \Q^{*2}$. By looking at the column of the conjugacy class $2A$ in Table \ref{taud8p}, one sees that either $\varrho_{9}$ or $\varrho_{10}$ is a constituent of $\theta_C$, since otherwise~$\theta_C$ would not be faithful. Since $\varrho_{9}=\overline\varrho_{10}$, we deduce that $\theta_C=\varrho_9+\varrho_{10}$. Moreover, by Lemma \ref{restheta}, $\Res^{G_C}_{H_C}\theta_C=2\cdot\varrho$, where $\varrho$ is a representation of $H_C\simeq D_8$. Since the only faithful representation of $D_8$ is irreducible, it follows that $\Res^{G_C}_{H_C} \varrho_9=\Res^{G_C}_{H_C}\varrho_{10}=\varrho$. The case $2u\in \Q^{*2}$ is analogous.
\end{proof}
Proposition \ref{thetad8} and Theorem \ref{supguap} implie the following result.
\begin{coro}\label{coro: D8} If $C'$ is a twist of $C$ such that $V_\ell(C')$ is a simple $\Q_\ell[G_K]$-module, then $$\theta(C,C')\otimes V_\ell(C)\simeq
\begin{cases}
\Q[\Gal(K/\Q)]\otimes V_\ell(C') & \text{if $u$ and $2u\not\in \Q^{*2}$.}\\
2\cdot\Q[\Gal(K/\Q)]\otimes V_\ell(C') & \text{if $u$ or $2u\in \Q^{*2}$.}
\end{cases}
$$
\end{coro}
\begin{proof} If $u\in \Q^{*2}$, the fact that $\Tr\theta_C=\chi_{9}+\chi_{10}$ together with $g^2/2=[K\colon\Q]=2$, guarantees that we are in case (II) of Theorem \ref{supguap}. The case $2u\in \Q^{*2}$ is analogous. If $u$ and $2u\not\in \Q^{*2}$, then we are in case (I).
\end{proof}

\subsection{$\Aut(C)\simeq D_{12}$}\label{subd12}

\begin{prop}[Proposition 2.2 of \cite{CQ07}] There is a bijection between the $\overline \Q$-isomorphism classes of genus $2$ curves defined over $\Q$ with $\Aut(C)\simeq D_{12}$ and the open set of the affine line
$\Q^*\smallsetminus\{1/4,-1/50\}$, given by associating to each $u\in \Q^*\smallsetminus\{1/4,-1/50\}$ the projective curve of equation
$$Y^2Z^4=X^6+X^3Z^3+uZ^6\,.$$
\end{prop}
As follows from Proposition 4.9 of \cite{CQ07}, the curve of the previous proposition is $\overline \Q$-isomorphic to
\begin{equation}\label{eqd12}
\begin{array}{ll}
C=C_u\colon Y^2Z^4= & \displaystyle{27 u X^6-2916 u^2 X^5 Z+243 u^2 X^4 Z^2 +29160 u^3 X^3Z^3+ }\\
 & \displaystyle{729 u^3 X^2 Z^4-26244 u^4 X Z^5+729 u^4 Z^6\,.}\\
\end{array}
\end{equation}
This curve corresponds to the curve appearing in Proposition 4.9 of \cite{CQ07}, with choice of parameters $z=0$, $s=u$ and $v=u/3$. Its group of automorphisms is computed loc. cit. in Proposition 3.5, and is generated by
$$
U=\begin{pmatrix}
0 & \sqrt{ u}/3\\
3/\sqrt{u} & 0
\end{pmatrix}\,,
\qquad
V=\begin{pmatrix}
1/2 & -\sqrt u/\sqrt{12}\\
3\sqrt 3/ \sqrt{4 u} & 1/2
\end{pmatrix}\,,
$$
from which we see that $K=\Q(\sqrt u,\sqrt 3)$ (observe the change of two signs in the matrix~$V$ with respect \cite{CQ07}). Note that $U$ and $V$ satisfy the relations $U^2=1$, $V^6=1$ and $UV=V^5U$. For the character table of the group $G_C$, see in Section \ref{Appendix} Table~\ref{taud12} if $u$ and $3u\not\in \Q^{*2}$; Table~\ref{taud12p} if $u\in \Q^{*2}$; and Table~\ref{taud12pp} if $3u\in \Q^{*2}$ .

\begin{prop}\label{thetad12} One has
$$\Tr\theta_C=
\begin{cases}\chi_{15} & \text{ if $u$ and $3u\not\in \Q^{*2}$,}\\
\chi_{i}+\chi_{j}, \text{ for $i\not =j\in\{10,11,12\}$}  & \text{ if $u\in \Q^{*2}$,}\\
\chi_{8}+\chi_{9}  & \text{ if $3u\in \Q^{*2}$.}\\
\end{cases}
$$
Moreover, $\Res^{G_C}_{H_C}\chi_i=\Res^{G_C}_{H_C}\chi_j$ in the second case, and $\Res^{G_C}_{H_C}\chi_8=\Res^{G_C}_{H_C}\chi_9$ in the third case.
\end{prop}

\begin{proof} The dimension of $\theta_C$ is $4$. Suppose that $u$ and $3u\not\in \Q^{*2}$. By Lemma \ref{hiper} below, and by looking at the column of the conjugacy class $2A$ in Table \ref{taud12}, one sees that $\varrho_{13}$, $\varrho_{14}$ and $\varrho_{15}$ are the only possible constituents of $\theta_C$. We deduce that $\theta_C\simeq \varrho_{15}$ from the fact that none of the representations $2\cdot\varrho_{13}$, $2\cdot\varrho_{14}$ and $\varrho_{13}\oplus\varrho_{14}$ is faithful.

One can also directly compute the representation $\theta_C$. Analogaously to the case $\Aut(C)\simeq D_8$ one has $\End^0_K(J(C))=\langle 1^*,\,U^*,\,V^*,\,U^*V^*\rangle_\Q$. Moreover, since the algebra $\langle 1^*,\,V^*\rangle$ has no zero divisors, one deduces that ${V^*}^2=V^*-1$. Let $\sigma,\,\tau\in\Gal(K/\Q)$ be such that $\sigma(\sqrt u)=-\sqrt u$ and $\tau(\sqrt 3)=-\sqrt 3$. Then ${}^\sigma U=UV^3$, ${}^{\sigma}V=V^5$, ${}^{\tau}U=U$, ${}^{\tau}V=V^5$.

Suppose that $u\in \Q^{*2}$. By Lemma \ref{restheta}, $\Res^{G_C}_{H_C}\theta_C=2\cdot\varrho$. The only faithful representation of $H_C\simeq D_{12}$ is irreducible. This, together with the fact that the dimension of the irreducible representations of $G_C$ is at most~$2$ (see Table~\ref{taud12p}), implies that $\theta_C$ is the sum of two irreducible representations of dimension~2. The only sums of two irreducible representations of dimension~2 of~$G_C$, which are faithful are $\chi_{10}+\chi_{11}$, $\chi_{11}+\chi_{12}$, or $\chi_{10}+\chi_{12}$. The case $3u\in \Q^{*2}$ is analogous.
\end{proof}

\begin{lema}\label{hiper} Let $C$ be a smooth projective hyperelliptic curve. Let $w$ be the hyperelliptic involution of $C$. Then, one has  $$\Tr\theta_C((w,\id))=-\dim\End_K^0(J(C))\,.$$
\end{lema}
\begin{proof} Observe that for $\psi\in\End_K^0(J(C))$, one has $\theta_C((w,\id))(\psi)=-\psi$.
\end{proof}

Proposition \ref{thetad12} and Theorem \ref{supguap} implie the following result.
\begin{coro}\label{coro: D12} If $C'$ is a twist of $C$ such that $V_\ell(C')$ is a simple $\Q_\ell[G_K]$-module, then $$\theta(C,C')\otimes V_\ell(C)\simeq
\begin{cases}
\Q[\Gal(K/\Q)]\otimes V_\ell(C') & \text{if $u$ and $3u\not\in \Q^{*2}$.}\\
2\cdot\Q[\Gal(K/\Q)]\otimes V_\ell(C') & \text{if $u$ or $3u\in \Q^{*2}$.}
\end{cases}
$$
\end{coro}
\begin{proof} If $u$ and $3u\not\in \Q^{*2}$, the fact that $\Tr\theta_C=\chi_{15}$ together with $g^2=[K\colon\Q]=4$, guarantees that we are in case (I) of Theorem \ref{supguap}. If $u$ or $3u\in \Q^{*2}$, then we are in case (II).
\end{proof}

\subsection{$L$-functions of twisted genus $2$ curves}

Now the proof of Theorem \ref{theorem: application} is immediate. If $p$ is an unramified prime in $L/\Q$, then the reciprocal of the characteristic polynomial of $\Frob_p$ acting on the $\Q_\ell[G_\Q]$-module at the left-hand side of the isomorphism of Corollary \ref{coro: D8} or Corollary \ref{coro: D12} is $L_{ p}(C/\Q,\theta_C\circ \lambda_\phi,T)$. Recall that $f$ denotes the residue class degree of $p$ in $K/\Q$. The result follows from the fact that the right-hand side of the isomorphism of Corollary \ref{coro: D8} or Corollary \ref{coro: D12} is of the form $\varrho\otimes V_\ell(C')$, where $\varrho$ is a 4-dimensional representation of $\Gal(K/\Q)$ such that $\varrho(\Frob_p)$ has four eigenvalues equal to 1 if $f=1$, and two eigenvalues equal to 1, and two equal to -1 if $f=2$.

Observe that thanks to Theorem \ref{theorem: application}, from the local factor $L_{p}(C/\Q,T)$ and the representation~$\theta(C,C')\simeq \theta_C\circ\lambda_\phi$, either the polynomial $L_{ p}(C'/\Q,T)$ or the product $L_{ p}(C'/\Q,T)\cdot L_{ p}(C'/\Q,-T)$ can be determined. The indeterminacy of the sign of~$a_p'$ which follows from the product $L_{ p}(C'/\Q,T)\cdot L_{ p}(C'/\Q,-T)$, can not be handled with the relation
$$\sgn(\Tr(\theta(C,C')(\Frob_{ p}))=\sgn(a_{ p}\cdot a_{ p}')\,.$$
from Proposition \ref{signe}, since this relation only holds for~$f=1$.

\newpage
\section{Appendix: Character tables of twisting groups}\label{Appendix}

In the following tables, the notation $\GAP(n,m)$ indicates the $m$-th group of order $n$ in the ordered list of finite groups of \cite{Gap}.

\begin{table}[h]
\scriptsize
$$
\begin{array}{r|rrrrrrrrrrr}
\rm{Class} &   1A & 2A & 2B & 2C & 2D & 2E & 4A & 4B & 4C & 8A & 8B  \\\hline
\rm{Size}  &   1 & 1 & 2 & 4 & 4 & 4 & 2 & 2 & 4 & 4 & 4\\ \hline
\chi_{ 1} & 1 & 1 & 1 & 1 & 1 & 1 & 1 & 1 & 1 & 1 & 1\\
\chi_{2} &  1 & 1 &-1 & 1 &-1 & 1 & 1 &-1 &-1 &-1 & 1\\
\chi_3 &  1 & 1 & 1 & 1& -1 &-1 & 1 & 1 & 1& -1& -1\\
\chi_4 &  1 & 1& -1 & 1 & 1 &-1 & 1 &-1 &-1 & 1 &-1\\
\chi_5 &  1&  1& -1 &-1 & 1 &-1 & 1 &-1 & 1& -1 & 1\\
\chi_6 &  1 & 1 & 1 &-1 &-1 &-1 & 1 & 1 &-1 & 1 & 1\\
\chi_7  & 1 & 1 &-1 &-1 &-1&  1&  1& -1&  1&  1& -1\\
\chi_8  & 1 & 1 & 1 &-1 & 1 & 1&  1 & 1 &-1 &-1 &-1\\
\chi_9  & 2 & 2  &2 & 0 & 0 & 0 &-2 &-2 & 0 & 0 & 0\\
\chi_{10} &  2 & 2 &-2 & 0 & 0 & 0 &-2&  2 & 0 & 0 & 0\\
\chi_{11} &  4& -4 & 0 & 0 & 0 & 0 & 0  &0 & 0 & 0 & 0\\
\end{array}
$$
\caption{Character table of $D_{8}\rtimes (C_2\times C_2)\simeq \GAP(32,43)$} \label{taud8}
\end{table}

\begin{table}[h]
\scriptsize
$$
\begin{array}{r|rrrrrrrrrr}
\rm{Class} &  1A & 2A & 2B & 2C & 2D & 4A & 4B & 4C & 4D & 4E  \\\hline
\rm{Size}  &  1 & 1 & 2 & 2 & 2 & 1 & 1 & 2 & 2 & 2 \\ \hline
\chi_1 & 1 & 1 & 1 & 1 & 1 & 1 & 1 & 1 & 1 & 1\\
\chi_2 & 1 & 1 &-1 & 1 & 1 & -1 & -1& 1 &-1 &-1\\
\chi_3 & 1 & 1 &-1 &-1 &-1 & -1 & -1 &  1 & 1 & 1\\
\chi_4 & 1 & 1 & 1 &-1 &-1 &  1 &  1 & 1 &-1 &-1\\
\chi_5 & 1 & 1 & 1 &-1 & 1 & -1 & -1 &-1 & 1 &-1\\
\chi_6 & 1 & 1 & 1 & 1 &-1 & -1 & -1 &-1 &-1 & 1\\
\chi_7 & 1 & 1 &-1 &-1 & 1 &  1 &  1 &-1 &-1 & 1\\
\chi_8 & 1 & 1 &-1 & 1 &-1 &  1 &  1 &-1  &1 &-1\\
\chi_9 & 2 &-2 & 0 & 0 & 0 & 2i & -2i & 0 & 0 & 0\\
\chi_{10} & 2 &-2 & 0  &0 & 0 & -2i & 2i & 0 & 0 & 0 \\
\end{array}
$$
\caption{Character table of $D_{8}\rtimes C_2\simeq \GAP(16,13)$} \label{taud8p}
\end{table}

\begin{table}[h]
\scriptsize
$$
\begin{array}{r|rrrrrrr}
\rm{Class} &  1A & 2A & 2B & 2C & 4A & 8A & 8B\\\hline
\rm{Size}  & 1 & 1 & 4 & 4 & 2 & 2 & 2 \\\hline
\chi_1  &  1 & 1 & 1 & 1 & 1 &  1 &  1\\
\chi_2  &  1 & 1 &-1 &-1 & 1  & 1 &  1\\
\chi_3 &  1 & 1 & -1 & 1 & 1 & -1 & -1\\
\chi_4 & 1 & 1 & 1& -1 & 1 & -1 & -1\\
\chi_5  &  2 & 2 & 0 & 0& -2 &  0 &  0\\
\chi_6  &  2& -2&  0 & 0 & 0 & \zeta_8 &-\zeta_8\\
\chi_7  &  2 &-2 & 0 & 0 & 0& -\zeta_8 & \zeta_8\\
\end{array}
$$
\caption{Character table of $D_{8}\rtimes C_2\simeq \GAP(16,7)$} \label{taud8pp}
\end{table}

\begin{table}
\scriptsize
$$
\begin{array}{r|rrrrrrrrrrrrrrr}
\rm{Class} &   1A & 2A & 2B & 2C & 2D & 2E & 2F & 2G & 3A & 4A & 4B & 6A & 6B & 6C & 12A \\\hline
\rm{Size}  &   1 & 1 & 2 & 2& 3 & 3 & 6 & 6 & 2&2 & 6&2 & 4&4&4 \\ \hline
\chi_{ 1} &   1 & 1 & 1 & 1 & 1 & 1 & 1 & 1 & 1 & 1 & 1 & 1 & 1 & 1 & 1\\
\chi_{ 1} &  1 & 1 & 1 &-1 &-1 &-1 &-1 & 1 & 1 &-1 & 1 & 1 & 1 &-1 &-1\\
\chi_{ 1} &  1 & 1 &-1 & 1 &-1 &-1 & 1&-1 & 1 &-1 & 1 & 1 &-1 & 1 &-1\\
\chi_{ 1} &   1 & 1 &-1 &-1 & 1 & 1 &-1 &-1 & 1 & 1 & 1 & 1& -1& -1&  1\\
\chi_{ 1} &   1 & 1 & 1 & 1 &-1 &-1& -1& -1&  1&  1 &-1 & 1 & 1 & 1 & 1\\
\chi_{ 1} &  1 & 1 & 1& -1&  1 & 1 & 1 &-1 & 1 &-1 &-1 & 1 & 1 &-1& -1\\
\chi_{ 1} &   1 & 1& -1&  1 & 1 & 1 &-1 & 1 & 1 &-1 &-1 & 1 &-1 & 1 &-1\\
\chi_{ 1} &   1  &1 &-1& -1& -1& -1&  1 & 1 & 1 & 1 &-1 & 1 &-1& -1 & 1\\
\chi_{ 1} &   2 & 2 & 2 & 2 & 0 & 0 & 0 & 0 &-1 & 2 & 0 &-1& -1& -1& -1\\
\chi_{ 10} &   2 & 2 &-2 &-2 & 0 & 0 & 0 & 0 &-1 & 2 & 0& -1&  1 & 1& -1\\
\chi_{ 11} &   2 & 2 & 2 &-2 & 0 & 0 & 0 & 0 &-1 &-2 & 0& -1 &-1 & 1  &1\\
\chi_{ 12} &   2 & 2 &-2 & 2 & 0 & 0 & 0 & 0 &-1 &-2 & 0 &-1 & 1 &-1 & 1\\
\chi_{ 13} &   2 &-2 & 0 & 0 &-2 & 2 & 0 & 0 & 2 & 0 & 0 &-2 & 0 & 0 & 0\\
\chi_{ 14} &   2 &-2 & 0 & 0 & 2 &-2  &0 & 0 & 2 & 0 & 0 &-2 & 0 & 0 & 0\\
\chi_{ 15} &   4 &-4 & 0 & 0 & 0 & 0  &0 & 0 &-2 & 0 & 0 & 2 & 0 & 0 & 0\\
\end{array}
$$
\caption{Character table of $D_{12}\rtimes (C_2\times C_2)\simeq \GAP(48,38)$} \label{taud12}

$$
\begin{array}{r|rrrrrrrrrrrr}
\rm{Size}  & 1 & 1 & 1 & 1 & 3 & 3 & 3 & 3 & 2 & 2 & 2 & 2 \\\hline
\rm{Class} &  1A & 2A & 2B & 2C & 2D & 2E & 2F & 2G & 3A & 6A & 6B & 6C\\\hline
\chi_1 & 1 & 1 & 1 & 1 & 1 & 1 & 1 & 1 & 1 & 1 & 1 & 1\\
\chi_2 & 1 &-1 &-1 & 1 & 1 &-1 &-1 & 1 & 1 &-1 & 1 &-1\\
\chi_3 & 1 &-1 & 1 &-1 &-1 &-1 & 1 & 1 & 1 &-1 &-1 & 1\\
\chi_4 &  1 & 1 &-1 &-1 &-1 & 1 &-1 & 1 & 1 & 1 &-1 &-1\\
\chi_5 & 1 & 1 & 1 & 1 &-1 &-1 &-1 &-1 & 1 & 1 & 1 & 1\\
\chi_6 & 1 &-1 &-1 & 1 &-1 & 1 & 1 &-1 & 1 &-1 & 1 &-1\\
\chi_7 & 1 &-1 & 1 &-1 & 1 & 1 &-1 &-1 & 1 &-1 &-1 & 1\\
\chi_8 & 1 & 1 &-1 &-1 & 1 &-1 & 1 &-1 & 1 & 1 &-1 &-1\\
\chi_9 & 2 & 2 & 2 & 2 & 0 & 0 & 0 & 0 &-1 &-1 &-1 &-1\\
\chi_{10} & 2 &-2 &-2 & 2 & 0 & 0 & 0 & 0 &-1 & 1 &-1 & 1\\
\chi_{11} & 2 & 2 &-2 &-2 & 0 & 0 & 0 & 0 &-1 &-1 & 1 & 1\\
\chi_{12} & 2 &-2 & 2 &-2 & 0 & 0 & 0 & 0 &-1 & 1 & 1& -1\\
\end{array}
$$
\caption{Character table of $D_{12}\rtimes C_2\simeq \GAP(24,14)$} \label{taud12p}

$$
\begin{array}{r|rrrrrrrrr}
\rm{Class}  & 1A & 2A & 2B & 2C & 3A & 4A & 6A & 6B & 6C\\\hline
\rm{Size}   & 1 & 1 & 2 & 6 & 2 & 6 & 2 & 2 & 2\\\hline
\chi_1 & 1 & 1 & 1 & 1 & 1 & 1 & 1  &  1&  1 \\
\chi_2 & 1 & 1 & 1 &-1 & 1 &-1 &    1  &   1 & 1\\
\chi_3 & 1 & 1 &-1 &-1 & 1 & 1 &   -1  &  -1 & 1\\
\chi_4 & 1 & 1 &-1 & 1 & 1 &-1 &   -1  &  -1 & 1\\
\chi_5 & 2 & 2 &-2 & 0 &-1 & 0 &    1  &   1 &-1\\
\chi_6 & 2 &-2 & 0 & 0 & 2 & 0 &    0  &   0 &-2\\
\chi_7 & 2 & 2 & 2 & 0 &-1 & 0 &   -1  &  -1 &-1\\
\chi_8 & 2 &-2 & 0 & 0 &-1 & 0 & -\sqrt{-3}&  \sqrt{-3} & 1\\
\chi_9 & 2 &-2 & 0 & 0 &-1 & 0 & \sqrt{-3} &  -\sqrt{-3} & 1\\
\end{array}
$$
\caption{Character table of $D_{12}\rtimes C_2\simeq \GAP(24,8)$} \label{taud12pp}
\end{table}

\end{document}